\documentclass{SCAE}%SCAEOL for online version; SCAE for publication version; SCAES for the paper dedicated to somebody.
\numberwithin{equation}{section}
\def\pr{\textsf{P}} % the symbol P for probability used the sans serif letter
\def\ep{\textsf{E}} % the symbol E for expectation used the sans serif letter
\def\Sbep{\widehat{\mathbb E}} % the symbol E for sub-linear expectation
\def\cSbep{\widehat{\mathcal E}} % the symbol E for conjugate sub-linear expectation
\def\extSbep{\mathbb E^{\ast}} % the symbol  for out sub-linear expectation
\def\Capc{\mathbb V} % the symbol V for capacity under E
\def\cCapc{\mathcal V} % the symbol C for conjugate capacity
\def\outCapc{\mathbb V^{\ast}}
\def\outcCapc{\mathcal V^{\ast}}
\begin{document}

%Basic Information
\Year{2016} %
\Month{April}
\Vol{59} %
\No{4} %
\BeginPage{751} %
\EndPage{768} %
\AuthorMark{Zhang L X}
\ReceivedDay{August 22, 2014}
\AcceptedDay{July 14, 2015}
\PublishedOnlineDay{; published online  December 23, 2015}
\DOI{10.1007/s11425-015-5105-2} % The author doesn't need fill in it.

% \title[short text for running head]{full title}{comments for title}
\title{ Rosenthal's inequalities for independent and negatively dependent random variables
  under sub-linear expectations with applications}{}

% \author[]{Full name}{footnote}
% Remark:  One \author for one author

\author[1]{ZHANG Li-Xin}{}

\address[{\rm1}]{ School of Mathematics, Zhejiang University, Hangzhou {\rm 310027}, P. R. China}
\Emails{stazlx@zju.edu.cn}\maketitle

%     Abstract is required.

 {\begin{center}
\parbox{14.5cm}{\begin{abstract}
Classical Kolmogorov's and Rosenthal's inequalities
for the maximum partial sums of random variables are basic tools for studying the strong laws of large numbers.
In this paper, motived by the notion of independent and identically distributed  random variables
under the sub-linear expectation  initiated by Peng\cite{Peng2006}\cite{Peng2008b}, we introduce the concept of negative dependence of random variables and establish
Kolmogorov's and Rosenthal's inequalities  for the maximum partial sums of negatively  dependent random variables under the sub-linear expectations.
As an application, we show that Kolmogorov's strong law of larger numbers holds for independent and identically distributed random variables
under a continuous sub-linear expectation if and only if
the corresponding Choquet integral is finite.
\vspace{-3mm}
\end{abstract}}\end{center}}

%  Keyword is required.

%  \subjclass is required.

%%%%%%%%%%%%%%%%%%%%%%%%%%%%%%%%%%%%%%%%%%%%%%%%%%%%%%%%%%%%
\renewcommand{\baselinestretch}{1.2}
  \renewcommand{\arraystretch}{1.5}

%%%%%%%%%%%%%%%%%%%%%%%%%%%%%%%%%%%%%%%%%%%%%%%%%%%%%%%%%%%%
%% Text of article.
%%%%%%%%%%%%%%%%%%%%%%%%%%%%%%%%%%%%%%%%%%%%%%%%%%%%%%%%%%%%
%    Section headings
\baselineskip 11pt\parindent=10.8pt  \wuhao

\section{ Introduction and notations.}
\setcounter{equation}{0}

Non-additive probabilities and non-additive expectations  are useful tools for studying  uncertainties in statistics, measures of risk,
superhedging in finance and non-linear stochastic calculus, c.f. Denis and  Martini\cite{DM2006}, Gilboa\cite{Gilboa1987},
Marinacci\cite{Marinacci1999},  Peng\cite{Peng1997}\cite{Peng1999}\cite{Peng2006}\cite{Peng2008a} etc.
This paper considers the  general sub-linear expectations and related non-additive probabilities generated by them.
The notion of independent and identically distributed random variables under  the sub-linear expectations is introduced by  Peng\cite{Peng2006}\cite{Peng2008b} and
the weak convergence such as  central limit theorems and weak laws of large numbers are studied.
Because   the proofs of classical Kolmogorov's inequalities and Rosenthal's inequalities
for the maximum partial sums of random variables depend basically on the additivity of the probabilities and the expectations,
 such inequalities have not been established under  the sub-linear expectations.
 As a result, very few results on strong laws of larger numbers   are found under the sub-linear expectations.
  Recently, Chen\cite{Chen2010} obtained Kolmogorov's strong law of larger numbers for i.i.d.
   random variables under the condition of finite $(1+\epsilon)$-moments by establishing an inequality of an exponential moment of partial sums of
  truncated independent random variables. The moment condition is much stronger than the one for the classical   Kolmogorov strong law of larger numbers.
  Also, Gao and  Xu\cite{GaoXu2011}\cite{GaoXu2012} studied the large deviations and moderate deviations for quasi-continuous random variables in a
complete separable metric space under the Choquet capacity generalized by a regular sub-linear expectation.
    The main purpose of this paper is to establish basic inequalities for the maximum partial sums of independent random variables
    in the general sub-linear expectation spaces.
    These inequalities are basic tools to study the strong limit theorems. They are also essential tools to prove the functional central limit theorem (c.f. Zhang\cite{Zhang2015}).  In the remainder of this section, we give some notations under the sub-linear expectations.
  For explaining our main idea, we prove Kolmogorov's inequality as our first result. And then, we introduce the concept of negative dependence under the sub-linear expectation
  which is  an extension of independence as well as the classical negative dependence.
   In the next section,  we establish Rosenthal's inequalities for this kind of negatively dependent random variables.
   In Section 3, as applications of these inequalities, we establish the Kolmogorov
   type strong laws of large numbers under the weakest moment conditions. In particular, we show that Kolmogorov's
   type strong law of large numbers holds for independent and identically distributed random variables under a continuous sub-linear expectation if and only if
   the the corresponding Choquet integral is finite.

We use the notations of Peng\cite{Peng2008b}. Let  $(\Omega,\mathcal F)$
 be a given measurable space  and let $\mathscr{H}$ be a linear space of real functions
defined on $(\Omega,\mathcal F)$ such that if $X_1,\ldots, X_n \in \mathscr{H}$  then $\varphi(X_1,\ldots,X_n)\in \mathscr{H}$ for each
$\varphi\in C_{l,Lip}(\mathbb R_n)$,  where $C_{l,Lip}(\mathbb R_n)$ denotes the linear space of (local Lipschitz)
functions $\varphi$ satisfying
\begin{eqnarray*} & |\varphi(\bm x) - \varphi(\bm y)| \le  C(1 + |\bm x|^m + |\bm y|^m)|\bm x- \bm y|, \;\; \forall \bm x, \bm y \in \mathbb R_n,&\\
& \text {for some }  C > 0, m \in \mathbb  N \text{ depending on } \varphi. &
\end{eqnarray*}
$\mathscr{H}$ is considered as a space of ``random variables''. In this case we denote $X\in \mathscr{H}$.
\begin{remark} It is easily seen that if $\varphi_1,\varphi_2\in C_{l,Lip}(\mathbb R_n)$, then $\varphi_1\vee \varphi_2, \varphi_1\wedge \varphi_2\in C_{l,Lip}(\mathbb R_n)$ because $\varphi_1\vee \varphi_2=\frac{1}{2}( \varphi_1+ \varphi_2+|\varphi_1- \varphi_2|)$, $\varphi_1\wedge \varphi_2=\frac{1}{2}( \varphi_1+ \varphi_2-|\varphi_1- \varphi_2|)$.
\end{remark}

\begin{definition} A {\bf sub-linear expectation} $\Sbep$ on $\mathscr{H}$  is a functional $\Sbep: \mathscr{H}\to \overline{\mathbb R}:=[-\infty,\infty]$ satisfying the following properties: for all $X, Y \in \mathscr H$, we have
\begin{description}
  \item[\rm (a)] {\bf Monotonicity}: If $X \ge  Y$ then $\Sbep [X]\ge \Sbep [Y]$;
\item[\rm (b)] {\bf Constant preserving} : $\Sbep [c] = c$;
\item[\rm (c)] {\bf Sub-additivity}: $\Sbep[X+Y]\le \Sbep [X] +\Sbep [Y ]$  whenever $\Sbep [X] +\Sbep [Y ]$ is not of the form $+\infty-\infty$ or $-\infty+\infty$;
\item[\rm (d)] {\bf Positive homogeneity}: $\Sbep [\lambda X] = \lambda \Sbep  [X]$, $\lambda\ge 0$.
 \end{description}
 The triple $(\Omega, \mathscr{H}, \Sbep)$ is called a sub-linear expectation space. Give a sub-linear expectation $\Sbep $, let us denote the conjugate expectation $\cSbep$of $\Sbep$ by
$$ \cSbep[X]:=-\Sbep[-X], \;\; \forall X\in \mathscr{H}. $$
Obviously, for all $X\in \mathscr{H}$, $\cSbep[X]\le \Sbep[X]$. We also call $\Sbep[X]$ and $\cSbep[X]$ the upper-expectation and lower-expectation of $X$ respectively.
\end{definition}

\begin{definition} ({\em Peng\cite{Peng2006}\cite{Peng2008b}})
\begin{description}
  \item[ \rm (i)] ({\bf Identical distribution}) Let $\bm X_1$ and $\bm X_2$ be two $n$-dimensional random vectors defined
respectively in sub-linear expectation spaces $(\Omega_1, \mathscr{H}_1, \Sbep_1)$
  and $(\Omega_2, \mathscr{H}_2, \Sbep_2)$. They are called identically distributed, denoted by $\bm X_1\overset{d}= \bm X_2$  if
$$ \Sbep_1[\varphi(\bm X_1)]=\Sbep_2[\varphi(\bm X_2)], \;\; \forall \varphi\in C_{l,Lip}(\mathbb R_n), $$
whenever the sub-expectations are finite.
\item[\rm (ii)] ({\bf Independence})   In a sub-linear expectation space  $(\Omega, \mathscr{H}, \Sbep)$, a random vector $\bm Y =
(Y_1, \ldots, Y_n)$, $Y_i \in \mathscr{H}$ is said to be independent to another random vector $\bm X =
(X_1, \ldots, X_m)$ , $X_i \in \mathscr{H}$ under $\Sbep$  if for each test function $\varphi\in C_{l,Lip}(\mathbb R_m \times \mathbb R_n)$
we have
$$ \Sbep [\varphi(\bm X, \bm Y )] = \Sbep \big[\Sbep[\varphi(\bm x, \bm Y )]\big|_{\bm x=\bm X}\big],$$
whenever $\overline{\varphi}(\bm x):=\Sbep\left[|\varphi(\bm x, \bm Y )|\right]<\infty$ for all $\bm x$ and
 $\Sbep\left[|\overline{\varphi}(\bm X)|\right]<\infty$.
 \item[\rm (iii)] ({\bf IID random variables}) A sequence of random variables $\{X_n; n\ge 1\}$ is said to be independent, if
$X_{i+1}$ is independent to $(X_1,\ldots, X_i)$ for each $i\ge 1$. It is said to be identically distributed, if
 $X_i\overset{d}=X_1$   for each $i\ge 1$.
 \end{description}
\end{definition}

As shown by Peng\cite{Peng2008b}, it is important to note that under sub-linear expectations
the condition that ``$\bm Y$ is independent to $\bm X$'' does not implies automatically that ``$\bm X$
is independent to $\bm Y$''.

From the definition of independence, it is easily seen that, if $Y$ is independent to $X$ and $X\ge 0, \Sbep [Y]\ge 0$, then
\begin{equation}\label{eq1.1} \Sbep[XY]=\Sbep[X]\Sbep[Y].
\end{equation}
Further,  if $Y$ is independent to $X$ and $X\ge 0, Y\ge 0$, then
\begin{equation}\label{eq1.2} \Sbep[XY]=\Sbep[X]\Sbep[Y], \;\; \cSbep[XY]=\cSbep[X]\cSbep[Y].
\end{equation}

If  $\{X_n;n\ge 1\}$ is a sequence of independent random variables with both the upper-expectations $\Sbep[X_i]$ and lower-expectations $\cSbep[X_i]$ being zeros, then it is easily checked that
$$ \Sbep\Big[S_n^2\Big]=\Sbep\Big[\sum_{k=1}^nX_k^2+\sum_{i\ne j} X_iX_j\Big]=\sum_{k=1}^n\Sbep[X_k^2], $$
because $\Sbep[X_iX_j]=\cSbep[X_iX_j]=0$ for $i\ne j$ by the definition of the independence, where $S_n=\sum_{k=1}^nX_k$.  However, when the popular truncation method  is used for studying the limit theorems, the truncated random variables usually no longer have zero sub-linear expectations. It is hard to centralize a random variable such that its  upper-expectation  and lower-expectation are both zeros. But is is easy to centralize a random variable $X$ such that one of $\Sbep[X]$ and $\cSbep[X]$ is zero. For example,  the random variable $X-\Sbep[X]$ has  zero  upper-expectation.
So, the moments of $S_n$ with the condition $\Sbep[X_i]=0$ ($i=1,\ldots,n$) are much useful than those with the condition $\Sbep[X_i]=\cSbep[X_i]=0$ ($i=1,\ldots,n$). Unfortunately, by noting that the independence of  $X$ and $Y$ does not imply
$ \Sbep\left[\big(X-\Sbep[X]\big)\big(Y-\Sbep[Y]\big)\right]= 0 \;\; (\text{or } \le 0),$
even to get a good estimate of the second order moment $\Sbep\Big[\big(\sum_{k=1}^n\big(X_k-\Sbep[X_k]\big)\big)^2\Big]$ is  not a trivial work. As for the probability inequalities or moment inequalities for the   maximum partial sums $\max_{k\le n} S_k$, in the classical probability space, the proof  depends basically on the additivity of the probabilities and the expectations. For example, the integral on the event $\{\max_{i\le n} S_i\ge x\}$ is usually split to integrals on $\{\max_{i\le k} S_i< x, S_k\ge x\}$, $k=1,\ldots,n$.  The methods based on the additivity can not be used   under the framework of sub-linear expectations. Other popular techniques such as the symmetrization, the martingale method and the stopping time method   are  also not available under the sub-linear expectations because they are essentially based on the additivity property.  The main purpose of this paper is to establish the moment inequalities for  $\max_{k\le n} S_k$ which can be applied to truncated random variables freely. To explain our main idea, we first give the following result on  Kolmogorov's  inequality.

\begin{theorem}\label{th1} ({\em  Kolmogorov's  inequality}) Let $\{X_1,\ldots, X_n\}$
be a sequence  of random variables in $(\Omega, \mathscr{H}, \Sbep)$ with $\Sbep[X_k]=0$, $k=1,\ldots, n$.
Suppose that $X_k$ is independent to $(X_{k+1},\ldots, X_n)$ for each $k=1,\ldots, n-1$. Denote $S_k=X_1+\cdots+X_k$, $S_0=0$. Then
\begin{equation}\label{eqth1.1}
\Sbep\big[\big(\max_{k\le n}S_k\big)^2\big]\le \sum_{k=1}^n \Sbep[X_k^2].
\end{equation}
In particular,
$$ \Sbep\big[\big(S_n^+\big)^2\big]\le \sum_{k=1}^n \Sbep[X_k^2].$$
\end{theorem}
\begin{proof}  Set $T_k=\max\big(X_k,X_k+X_{k+1},\ldots, X_k+\cdots+X_n\big)$. Then $T_k, T_k^+\in \mathscr{H}$, and $T_k=X_k+T_{k+1}^+$, $T_k^2=X_k^2+2X_kT_{k+1}^+ +(T_{k+1}^+)^2$. It follows that
$$ \Sbep [T_k^2]\le \Sbep[X_k^2]+2\Sbep[X_k T_{k+1}^+]+\Sbep[(T_{k+1}^+)^2]. $$
Note $\Sbep[X_k T_{k+1}^+]=0$ by (\ref{eq1.1}). We conclude that
$$ \Sbep [T_k^2]\le \Sbep[X_k^2]+ \Sbep[(T_{k+1}^+)^2]\le  \Sbep[X_k^2]+ \Sbep[T_{k+1}^2]. $$
Hence
$ \Sbep [T_1^2]\le \sum_{k=1}^n \Sbep[X_k^2]. $
The proof is completed.
\end{proof}

In the above proof, the independence is utilized to get $\Sbep[X_k T_{k+1}^+]\le 0$ and so can be weakened. Recall that in the probability
$(\Omega, \mathcal F, \pr)$, two random vectors $\bm Y =
(Y_1, \ldots, Y_n)$  and $\bm X =
(X_1, \ldots, X_m)$ are said to be negatively dependent  if for each pair of coordinatewise nondecreasing (resp. non-increasing) functions
 $\varphi_1(\bm x)$ and $\varphi_2(\bm y)$ we have
$$ E_{\pr} [\varphi_1(\bm X)\varphi_2(\bm Y )] \le  E_{\pr} [\varphi_1(\bm X)] E_{\pr}[\varphi_2(\bm Y )]$$
whenever the expectations considered exist.

We introduce the concept of negative dependence under the sub-linear expectation.

 \begin{definition} ({\bf Negative dependence})  In a sub-linear expectation space  $(\Omega, \mathscr{H}, \Sbep)$, a random vector $\bm Y =
(Y_1, \ldots, Y_n)$, $Y_i \in \mathscr{H}$ is said to be negatively dependent to another random vector $\bm X =
(X_1, \ldots, X_m)$, $X_i \in \mathscr{H}$ under $\Sbep$  if for each pair of   test functions $\varphi_1\in C_{l,Lip}(\mathbb R_m)$ and
$\varphi_2\in C_{l,Lip}(\mathbb R_n)$
we have
$$ \Sbep [\varphi_1(\bm X)\varphi_2(\bm Y )] \le  \Sbep [\varphi_1(\bm X)]\Sbep[\varphi_2(\bm Y )]$$
whenever $\varphi_1(\bm X)\ge 0$, $\Sbep[\varphi_2(\bm Y )]\ge 0$, $\Sbep [|\varphi_1(\bm X)\varphi_2(\bm Y )|]<\infty$,
$\Sbep [|\varphi_1(\bm X)|]<\infty$,
  $\Sbep [|\varphi_2(\bm Y )|]<\infty$,
   and either $\varphi_1, \varphi_2$ are coordinatewise nondecreasing or $\varphi_1, \varphi_2$ are coordinatewise non-increasing.
\end{definition}

By the definition, it is easily seen that, if $\bm Y =
(Y_1, \ldots, Y_n)$  is   negatively dependent to $\bm X=(X_1, \ldots, X_m)$, $\varphi_1\in C_{l,Lip}(\mathbb R_m)$ and
$\varphi_2\in C_{l,Lip}(\mathbb R_n)$ are coordinatewise nondecreasing (resp.  non-increasing) functions, then
$\varphi_2(\bm Y )$ is negatively dependent to $\varphi_1(\bm X)$.  Further, if $Y\in \mathscr{H}$ is negatively dependent to $X\in \mathscr{H}$
and $X\ge 0$, $\Sbep[X]<\infty$, $\Sbep[|Y|]<\infty$, $\Sbep[Y]\le 0$, then
$$ \Sbep[YX]\le \Sbep\big[(Y-\Sbep[Y])X\big]+\Sbep\big[\Sbep[Y]X\big]\le \Sbep\big[ Y-\Sbep[Y] \big]\Sbep[X]\le 0. $$

It is obvious that, if $\bm Y$ is independent to $\bm X$, then $\bm Y$ is negatively dependent to $\bm X$. The following is the classical example introduced by
Huber and Strassen\cite{HuberStrassen1973}.

\begin{example} \label{example1} Let $\mathcal P$ be a family of probability measures defined on $(\Omega, \mathcal F)$. For any random variable $\xi$, we denote the upper expectation by
$$ \Sbep[\xi]=\sup_{Q\in \mathcal P}\ep_Q[\xi]. $$
Then $\Sbep[\cdot]$ is a sub-linear expectation. Moreover, if $\bm X$ and $\bm Y$ are independent under each $Q\in \mathcal P$, then
$\bm Y$ is negatively dependent to $\bm X$ under $\Sbep$. In fact,
\begin{align*}
 \Sbep [\varphi_1(\bm X)\varphi_2(\bm Y )]
=& \sup_{Q\in \mathcal P}\ep_Q [\varphi_1(\bm X)\varphi_2(\bm Y )] =\sup_{Q\in \mathcal P}\ep_Q [\varphi_1(\bm X)]\ep_Q[\varphi_2(\bm Y )] \\
\le & \sup_{Q\in \mathcal P}\ep_Q [\varphi_1(\bm X)]\sup_{Q\in \mathcal P}\ep_Q [\varphi_2(\bm Y )]=  \Sbep [\varphi_1(\bm X)]\Sbep[\varphi_2(\bm Y )]
\end{align*}
whenever $\varphi_1(\bm X)\ge 0$ and $\Sbep[\varphi_2(\bm Y )]\ge 0$.

However, $\bm Y$ may be not independent to $\bm X$.

With the similar argument, we can show that $\bm Y$ is negatively dependent to $\bm X$ under $\Sbep$ if $\bm X$ and $\bm Y$ are negatively dependent under each $Q\in \mathcal P$.
\end{example}

According to its proof, the conclusion of Theorem~\ref{th1} remains true under the concept of negative dependence.
\begin{corollary} Let $\{X_1,\ldots, X_n\}$ be a sequence  of random variables in $(\Omega, \mathscr{H}, \Sbep)$ with $\Sbep[X_k]\le 0$, $k=1,\ldots, n$. Suppose that $X_k$ is negatively dependent to $(X_{k+1},\ldots, X_n)$ for each $k=1,\ldots, n-1$.  Then (\ref{eqth1.1}) holds.
\end{corollary}

Our basic idea for obtaining Theorem \ref{th1} comes from Newman and Wright \cite{NewmanWright1981} and Matula\cite{Matula1992} where
 Kolmogorov's  inequality is estiblished for the classical positively and negatively dependent random variables respectively.

%%%%%%%%%%%%%%%%%%%%%%%%%%%%%%%%%%%%%%%%%%%%%%%%%%%%%%%%%%%%%%%%%%%%%%%%%%%%%%%%%%%%%%%%%%%%%%%%%%%%%%%%%%%%%%%%%%%
%%%%%%%%%%%%% Section 2 %%%%%%%%%%%%%%%%%%%%%%%%%%%%%%%%%%%

 \section{Rosenthal's inequalities}\label{secction2}

In this section, we extend Kolmogorov's  inequality to Rosenthal's inequalities.
For  moment  inequalities of partial sums of the classical negatively dependent random variables and related strong limit theorems,
one can refer to Shao\cite{Shao2000},  Su, Zhao and Wang \cite{SuZhaoWang1997},   Yuan  and An\cite{YuanAn2009},  Zhang\cite{Zhang2000}\cite{Zhang2001a}\cite{Zhang2001b},   Zhang and Wen\cite{ZhangWen2001} etc.
 Some   techniques  from these papers  will be used in the lines
of our  proofs. We let $\{X_1,\ldots, X_n\}$ be a sequence  of random variables in
 $(\Omega, \mathscr{H}, \Sbep)$, and denote $S_k=X_1+\ldots+ X_k$, $S_0=0$.

\begin{theorem}\label{th2} ({\em Rosnethal's inequality})
 (a) Suppose   that $X_k$ is negatively dependent to $(X_{k+1},\ldots, X_n)$ for each $k=1,\ldots, n-1$, and  $\Sbep[X_k]\le 0$, $k=1,\ldots, n$.  Then
\begin{equation}\label{eqth2.1}
\Sbep\left[\left|\max_{k\le n} S_k\right|^p\right]\le 2^{2-p}\sum_{k=1}^n \Sbep [|X_k|^p], \;\; \text{ for } 1\le p\le 2
\end{equation}
and
\begin{equation}\label{eqth2.1.2}
\Sbep\left[\left|\max_{k\le n} S_k\right|^p\right]\le C_p n^{p/2-1}\sum_{k=1}^n \Sbep [|X_k|^p], \;\; \text{ for }   p\ge 2.
\end{equation}

(b) Suppose  that $X_k$ is independent to $(X_{k+1},\ldots, X_n)$ for each $k=1,\ldots, n-1$, and  $\Sbep[X_k]\le 0$, $k=1,\ldots, n$. Then
\begin{equation}\label{eqth2.2}
\Sbep\left[\left|\max_{k\le n} S_k\right|^p\right]\le C_p\left\{ \sum_{k=1}^n \Sbep [|X_k|^p]+\left(\sum_{k=1}^n \Sbep [|X_k|^2]\right)^{p/2}\right\}, \;\; \text{ for }   p\ge 2.
\end{equation}

(c) In general, suppose   that $X_k$ is negatively dependent to $(X_{k+1},\ldots, X_n)$ for each $k=1,\ldots, n-1$,
 or $X_{k+1}$ is negatively dependent to $(X_1,\ldots, X_k)$ for each $k=1,\ldots, n-1$.  Then
\begin{align}\label{eqth2.3}
\Sbep\left[\max_{k\le n} \left|S_k\right|^p\right]\le  & C_p\left\{ \sum_{k=1}^n \Sbep [|X_k|^p]+\left(\sum_{k=1}^n \Sbep [|X_k|^2]\right)^{p/2} \right. \nonumber
 \\
& \qquad \left. +\left(\sum_{k=1}^n \big[\big(\cSbep [X_k]\big)^-+\big(\Sbep [X_k]\big)^+\big]\right)^{p}\right\}.
\end{align}
Here  $ C_p$ is a positive constant depending only on $p$.
\end{theorem}

If we consider the sequence $\{X_1,X_2,\ldots, X_n\}$ in the reverse order as $\{X_n,X_{n-1},\ldots, X_1\}$, by Theorem~\ref{th2} (a) and (b) we have the following corollary.
\begin{corollary}\label{cor2}   Let $\{X_1,\ldots, X_n\}$ be a sequence  of random variables in $(\Omega, \mathscr{H}, \Sbep)$
with $\Sbep[X_k]\le 0$, $k=1,\ldots, n$.

(a)  Suppose that $X_{k+1}$ is negatively dependent to $(X_1,\ldots, X_k)$ for each $k=1,\ldots, n-1$.  Then
\begin{equation}\label{eqcor2.1}
\Sbep\left[\left|\max_{k\le n} (S_n-S_k)\right|^p\right]\le 2^{2-p}\sum_{k=1}^n \Sbep [|X_k|^p], \;\; \text{ for } 1\le p\le 2
\end{equation}
and
\begin{equation}\label{eqcor2.2}
\Sbep\left[\left|\max_{k\le n}(S_n- S_k)\right|^p\right]\le C_p n^{p/2-1} \sum_{k=1}^n \Sbep [|X_k|^p], \;\; \text{ for }   p\ge 2.
\end{equation}
In particular,
\begin{equation}\label{eqcor2.3}
\Sbep\left[\left( S_n^+\right)^p\right]\le \begin{cases} 2^{2-p}\sum_{k=1}^n \Sbep [|X_k|^p], &  \text{ for } 1\le p\le 2, \\
C_p  n^{p/2-1} \sum_{k=1}^n \Sbep [|X_k|^p], &  \text{ for }  p\ge 2.
\end{cases}
\end{equation}

(b)  Suppose that $X_{k+1}$ is independent to $(X_1,\ldots, X_k)$ for each $k=1,\ldots, n-1$.  Then
\begin{equation}\label{eqcor2.2.1}
\Sbep\left[\left|\max_{k\le n}(S_n- S_k)\right|^p\right]\le C_p\left\{ \sum_{k=1}^n \Sbep [|X_k|^p]+\left(\sum_{k=1}^n \Sbep [|X_k|^2]\right)^{p/2}\right\}, \;\; \text{ for }   p\ge 2.
\end{equation}
In particular,
$$
\Sbep\left[\left( S_n^+\right)^p\right]\le C_p\left\{ \sum_{k=1}^n \Sbep [|X_k|^p]+\left(\sum_{k=1}^n \Sbep [|X_k|^2]\right)^{p/2}\right\}, \;\; \text{ for }   p\ge 2.
$$
\end{corollary}

For the moments under $\cSbep$, we have the following estimates.
\begin{theorem}\label{th2c}    Let $\{X_1,\ldots, X_n\}$ be a sequence  of  random variables in
 $(\Omega, \mathscr{H}, \Sbep)$ with $\cSbep[X_k]\le 0$, $k=1,\ldots, n$, and $1\le p\le 2$.
  If $X_k$ is  independent to $(X_{k+1},\ldots, X_n)$ for each $k=1,\ldots, n-1$,  then
\begin{equation}\label{eqth2c.1}
\cSbep\left[\left|\max_{k\le n} S_k\right|^p\right]\le 2^{2-p}\sum_{k=1}^n \Sbep [|X_k|^p], \;\; \text{ for } 1\le p\le 2.
\end{equation}
If $X_{k+1}$ is  independent to $(X_1,\ldots, X_k)$ for each $k=1,\ldots, n-1$, then
\begin{equation}\label{eqth2c.2}
\cSbep\left[\left|\max_{k\le n} (S_n-S_k)\right|^p\right]\le 2^{2-p}\sum_{k=1}^n \Sbep [|X_k|^p], \;\; \text{ for } 1\le p\le 2.
\end{equation}
\end{theorem}

To prove Theorems~\ref{th2}-\ref{th2c}, we need  H\"older's inequality under the sub-linear expectation which can be proved by the same may under the linear expectation due to the properties of
the monotonicity and sub-additivity (c.f. Proposition 1.16 of Peng\cite{Peng2010}).
 \begin{lemma} ({\em H\"older's inequality}) Let $p,q>1$ be two real numbers satisfying $\frac{1}p+\frac{1}{q}=1$. Then for two random variables  $X,Y$  in $(\Omega, \mathscr{H}, \Sbep)$ we have
 $$ \Sbep[|XY|]\le \left(\Sbep[|X|^p]\right)^{\frac{1}{p}}  \left(\Sbep[|Y|^p]\right)^{\frac{1}{q}}. $$
 \end{lemma}

{\em Proof of Theorem~\ref{th2}}.   Let $T_k$ be defined as in the proof of Theorem~\ref{th1}.

(a) We first prove (\ref{eqth2.1}).    Substituting $x=X_k$ and $y=T_{k+1}^+$ to the following elementary inequality
\begin{equation}\label{eqproofth2.1.0} |x+y|^p \le 2^{2-p}|x|^p+|y|^p+px|y|^{p-1}\text{sgn}(y), \;\;   1\le p\le 2
\end{equation}
yields
\begin{align*}
\Sbep\left[|T_k|^p\right]\le & 2^{2-p}\Sbep \left[|X_k|^p\right]+\Sbep \left[(T_{k+1}^+)^p\right]+p\Sbep \left[X_k(T_{k+1}^+)^{p-1}\right]\\
\le & 2^{2-p}\Sbep \left[|X_k|^p\right]+\Sbep \left[|T_{k+1}|^p\right]
\end{align*}
by the definition of negative dependence and the facts that $\Sbep \left[X_k\right]\le 0$, $T_{k+1}^+\ge 0$, and $T_{k+1}^+$ is
 a coordinatewise nondecreasing function of $X_{k+1},\ldots, X_n$. Hence
$$ \Sbep\left[|T_1|^p\right]\le   2^{2-p}\sum_{k=1}^{n-1}\Sbep \left[|X_k|^p\right]+\Sbep \left[|X_n|^p\right]. $$
So, (\ref{eqth2.1}) is proved.

For (\ref{eqth2.1.2}), by the   following elementary inequality
$$ |x+y|^p \le 2^{p}p^2|x|^p+|y|^p+px|y|^{p-1}\text{sgn}(y)+2^p p^2 x^2|y|^{p-2}, \;\;    p\ge 2, $$
we have
$$ |T_k|^p\le 2^pp^2|X_k|^p+|T_{k+1}|^p+pX_k(T_{k+1}^+)^{p-1}+2^p p^2X_k^2(T_{k+1}^+)^{p-2}. $$
  It follows that
\begin{equation}\label{eqproofth2.1.1}
 |T_i|^p\le 2^pp^2 \sum_{k=i}^n|X_k|^p+p\sum_{k=i}^{n-1} X_k(T_{k+1}^+)^{p-1}+2^p p^2\sum_{k=i}^{n-1}X_k^2(T_{k+1}^+)^{p-2}.
 \end{equation}
Hence by the definition of  the negative dependence and  H\"older's inequality,
\begin{align*}
\Sbep\left[|T_i|^p\right]\le & 2^pp^2\Sbep \left[\sum_{k=i}^n|X_k|^p\right]+p\sum_{k=i}^{n-1}\Sbep\left[ X_k(T_{k+1}^+)^{p-1}\right]
+2^p p^2 \sum_{k=i}^{n-1}\Sbep\left[ X_k^2 (T_{k+1}^+)^{p-2}\right]\\
\le & 2^pp^2\Sbep \left[\sum_{k=1}^n|X_k|^p\right]
+2^pp^2 \sum_{k=1}^{n-1}\left(\Sbep [|X_k|^p] \right)^{\frac{2}{p}} \left(\Sbep\left[|T_{k+1}|^p\right]\right)^{1-\frac{2}{p}}.
\end{align*}
Let $A_n=\max_{k\le n}\Sbep\left[|T_k|^p\right]$. Then
\begin{align*}
A_n\le
  2^pp^2 \sum_{k=1}^n\Sbep[|X_k|^p]
+2^pp^2 \sum_{k=1}^{n-1}\left(\Sbep [|X_k|^p] \right)^{\frac{2}{p}} A_n^{1-\frac{2}{p}}.
\end{align*}
From the above inequality, it can be shown that
\begin{align*}
A_n\le C_p\left\{\sum_{k=1}^n\Sbep[|X_k|^p]+\left(\sum_{k=1}^{n-1}\left(\Sbep [|X_k|^p] \right)^{\frac{2}{p}}\right)^{\frac{p}{2}}\right\}
\le C_p n^{p/2-1}\sum_{k=1}^n\Sbep[|X_k|^p].
\end{align*}
(\ref{eqth2.1.2}) is proved.

(b) Note the independence. From (\ref{eqproofth2.1.1}) it follows that
 \begin{align*}
\Sbep\left[|T_i|^p\right]\le & 2^pp^2\Sbep \left[\sum_{k=i}^n|X_k|^p\right]+p\sum_{k=i}^{n-1}\Sbep\left[ X_k(T_{k+1}^+)^{p-1}\right]
+2^p p^2 \sum_{k=i}^{n-1}\Sbep\left[ X_k^2 (T_{k+1}^+)^{p-2}\right]\\
= & 2^pp^2\Sbep \left[\sum_{k=i}^n|X_k|^p\right]+p\sum_{k=i}^{n-1}\Sbep[ X_k]\Sbep\left[(T_{k+1}^+)^{p-1}\right]+
 2^p p^2 \sum_{k=i}^{n-1}\Sbep [ X_k^2]\Sbep\left[(T_{k+1}^+)^{p-2}\right]\\
\le & 2^pp^2\Sbep \left[\sum_{k=1}^n|X_k|^p\right]
+2^pp^2 \sum_{k=1}^{n-1} \Sbep [X_k^2]   \left(\Sbep\left[|T_{k+1}|^p\right]\right)^{1-\frac{2}{p}}.
\end{align*}
Let $A_n=\max_{k\le n}\Sbep\left[|T_k|^p\right]$. Then
\begin{align*}
A_n\le
  2^pp^2 \Sbep \left[\sum_{k=1}^n|X_k|^p\right]
+2^pp^2 \sum_{k=1}^{n-1} \Sbep [X_k^2]  A_n^{1-\frac{2}{p}}.
\end{align*}
From the above inequality, it can be shown that
\begin{align*}
A_n\le C_p\left\{\sum_{k=1}^n\Sbep[|X_k|^p]+\left(\sum_{k=1}^n\Sbep [X_k^2] \right)^{\frac{p}{2}}\right\}.
\end{align*}
(\ref{eqth2.1.2}) is proved. $\Box$

(c) We first show
the  Marcinkiewicz-Zygmund inequality:
\begin{equation}\label{eqMZ}
\Sbep[\max_{k\le n}|S_k|^p]\le  C_p\left\{ \left( \sum_{k=1}^n\left(\big(\Sbep[X_k]\big)^++\big(\cSbep[X_k]\big)^- \right) \right)^p
 +    \Sbep\left(\sum_{k=1}^n X_k^2\right)^{\frac{p}{2}} \right\}.
\end{equation}
Without loss of generality, we assume  that $X_k$ is negatively dependent to $(X_{k+1},\ldots, X_n)$ for all $k=1,2,\ldots,n-1$. If $X_{k+1}$ is
negatively dependent to $(X_1,\ldots, X_k)$ for all $k=1,2,\ldots,n-1$, then (\ref{eqMZ}) will hold with $\max_{k\le n}|S_k|$ being replaced by
$\max_{0\le k\le n}|S_n-S_k|$. By noting  the fact $\max_{k\le n}|S_k|\le \max_{0\le k\le n}|S_n-S_k|+|S_n|\le 2\max_{0\le k\le n}|S_n-S_k|$,
 (\ref{eqMZ}) also is true.

Write   $\widetilde{T}_1=\max_{k\le n}|S_k|$. It is easily seen that $S_k+T_{k+1}^+=\max\big(S_k,S_{k+1},\ldots, S_n)\le T_1$.
So,  $T_{k+1}^+\le 2 \widetilde{T}_1$.
 Note (\ref{eqproofth2.1.1}). By the the definition of the negative dependence,
\begin{align*}\Sbep[X_k(T_{k+1}^+)^{p-1}]\le  & \begin{cases}  \Sbep[X_k]\Sbep[(T_{k+1}^+)^{p-1}], & \text{if }\Sbep[X_k] \ge 0 \\
 0, & \text{if }\Sbep[X_k] < 0 \end{cases}\\
 \le &  2^{p-1}\big(\Sbep[X_k]\big)^+ \Sbep[\widetilde{T}_1^{p-1}]\le   2^{p-1}\big(\Sbep[X_k]\big)^+ \big(\Sbep[\widetilde{T}_1^p]\big)^{1-\frac{1}{p}}
 \end{align*}
by the H\"older inequality.  By (\ref{eqproofth2.1.1}) and the H\"older inequality again, it follows that
 \begin{align*}
\Sbep\left[|T_1|^p\right]\le & 2^pp^2\Sbep \left[\sum_{k=1}^n|X_k|^p\right]+p\sum_{k=1}^{n-1}\Sbep\left[ X_k(T_{k+1}^+)^{p-1}\right]
+2^p p^2 \Sbep\left[\sum_{k=1}^{n-1}  X_k^2 (T_{k+1}^+)^{p-2}\right]\\
\le & 2^pp^2\Sbep \left[\sum_{k=1}^n|X_k|^p\right]+ 2^{p-1}p\sum_{k=1}^{n-1}\big(\Sbep[X_k]\big)^+ \big(\Sbep[\widetilde{T}_1^p]\big)^{1-\frac{1}{p}}\\
& + 2^p p^2 2^{p-2} \Sbep\left[\sum_{k=1}^{n-1} X_k^2 \widetilde{T}_1^{p-2}\right] \\
\le & 2^pp^2\Sbep \left[\sum_{k=1}^n|X_k|^p\right]+ 2^{p-1}p\left(\sum_{k=1}^{n-1}\big(\Sbep[X_k]\big)^+\right) \big(\Sbep[\widetilde{T}_1^p]\big)^{1-\frac{1}{p}}\\
& + 2^{2p-2}  p^2  \left[\Sbep\left(\sum_{k=1}^{n-1} X_k^2\right)^{\frac{p}{2}}\right]^{\frac{2}{p}}\big( \Sbep[\widetilde{T}_1^p]\big)^{1-\frac{2}{p}}.
\end{align*}
Similarly,
\begin{align*}
\Sbep\left[\left|\max_{k\le n}(-S_k)\right|^p\right]\le & 2^pp^2\Sbep \left[\sum_{k=1}^n|X_k|^p\right]
+ 2^{p-1}p\left(\sum_{k=1}^{n-1}\big(\Sbep[-X_k]\big)^+\right) \big(\Sbep[\widetilde{T}_1^p]\big)^{1-\frac{1}{p}}\\
& + 2^{2p-2}  p^2  \left[\Sbep\left(\sum_{k=1}^{n-1} X_k^2\right)^{\frac{p}{2}}\right]^{\frac{2}{p}}\big( \Sbep[\widetilde{T}_1^p]\big)^{1-\frac{2}{p}}.
\end{align*}
Hence
\begin{align*}
\Sbep[\widetilde{T}_1^p]\le & 2^{p+1}p^2\Sbep \left[\sum_{k=1}^n|X_k|^p\right]
+ 2^{p-1}p\left(\sum_{k=1}^n\big[\big(\Sbep[X_k]\big)^++\big(\cSbep[X_k]\big)^-\big]\right) \big(\Sbep[\widetilde{T}_1^p]\big)^{1-\frac{1}{p}}\\
& + 2^{2p-1}  p^2  \left[\Sbep\left(\sum_{k=1}^{n-1} X_k^2\right)^{\frac{p}{2}}\right]^{\frac{2}{p}}\big( \Sbep[\widetilde{T}_1^p]\big)^{1-\frac{2}{p}},
\end{align*}
which implies
\begin{align*}
\Sbep[\widetilde{T}_1^p]\le   C_p\left\{\Sbep \left[\sum_{k=1}^n|X_k|^p\right]
+ \left( \sum_{k=1}^n\left(\big(\Sbep[X_k]\big)^++\big(\cSbep[X_k]\big)^- \right) \right)^p
 +    \Sbep\left(\sum_{k=1}^n X_k^2\right)^{\frac{p}{2}} \right\}.
\end{align*}
Note
$$\left(\sum_{k=1}^n |x_k|^p\right)^{\frac{2}{p}}=\left(\sum_{k=1}^n |x_k^2|^{\frac{p}{2}}\right)^{\frac{2}{p}}
\le \sum_{k=1}^n x_k^2\;\; \text{ for } \frac{2}{p}\le 1. $$
So
$$ \sum_{k=1}^n |X_k|^p\le \left(\sum_{k=1}^n X_k^2\right)^{\frac{p}{2}}. $$
The Marcinkiewicz-Zygmund inequality (\ref{eqMZ}) is proved.

Now, for $2\le p\le 4$, applying (\ref{eqth2.1}) to the sequences $\{(X_1^+)^2,\ldots,(X_n^+)^2\}$   yields
 \begin{align*}
&\Sbep \Big[\Big(\Big\{\sum_{k=1}^n\big[(X_k^+)^2-\Sbep[(X_k^+)^2]\big] \Big\}^+\Big)^{\frac{p}{2}}\Big]\\
\le & 2^{2-\frac{p}{2}}\sum_{k=1}^n\Sbep\big[\big|(X_k^+)^2-\Sbep[(X_k^+)^2]\big|^{\frac{p}{2}}\big]
\le C_p\sum_{k=1}^n\Sbep\big[|X_k|^p\big].
\end{align*}
It follows that
\begin{align*}
 \Sbep\Big( \sum_{k=1}^n (X_k^+)^2 \Big)^{\frac{p}{2}}  \le C_p\left\{\Big( \sum_{k=1}^n \Sbep[(X_k^+)^2] \Big)^{\frac{p}{2}}
+\sum_{k=1}^n\Sbep\big[|X_k|^p\big]\right\}.
\end{align*}
Similarly
\begin{align*}
 \Sbep\Big( \sum_{k=1}^n (X_k^-)^2 \Big)^{\frac{p}{2}}  \le C_p\left\{\Big( \sum_{k=1}^n \Sbep[(X_k^-)^2] \Big)^{\frac{p}{2}}
+\sum_{k=1}^n\Sbep\big[|X_k|^p\big]\right\}.
\end{align*}
Hence
\begin{align*}
 \Sbep\Big( \sum_{k=1}^n  X_k^2 \Big)^{\frac{p}{2}}  \le C_p\left\{\Big( \sum_{k=1}^n \Sbep[ X_k^2] \Big)^{\frac{p}{2}}
+\sum_{k=1}^n\Sbep\big[|X_k|^p\big]\right\}.
\end{align*}
Substituting the above estimate to (\ref{eqMZ}) yield (\ref{eqth2.2}).

Suppose (\ref{eqth2.2}) is proved for $2^l< p\le 2^{l+1}$. Then applying it to the sequences $\{(X_1^+)^2,\ldots,(X_n^+)^2\}$ and
 $\{(X_1^-)^2,\ldots,(X_n^-)^2\}$ respectively
with $2^l< p/2\le 2^{l+1}$ yields
\begin{align*}
 \Sbep \Big[\Big(\sum_{k=1}^n(X_k^+)^2\Big)^{\frac{p}{2}}\Big]
\le & C_p\left\{\sum_{k=1}^n\Sbep\big[\big|(X_k^+)^2\big|^{\frac{p}{2}}\big]
+\Big(\sum_{k=1}^n\big(\Sbep[(X_k^+)^2]\big)^+\Big)^{\frac{p}{2}}
+ \left(\sum_{k=1}^n\Sbep\big[\big[(X_k^+)^2\big]^2\big]\right)^{\frac{p}{4}} \right\} \\
\le & C_p\left\{\sum_{k=1}^n\Sbep\big[\big|X_k\big|^p\big]+\left(\sum_{k=1}^n\Sbep\big[X_k^2\big]\right)^{\frac{p}{2}}
+\left(\sum_{k=1}^n\Sbep\big[X_k^4\big]\right)^{\frac{p}{4}}\right\}
\end{align*}
and
\begin{align*}
 \Sbep \Big[\Big(\sum_{k=1}^n(X_k^-)^2\Big)^{\frac{p}{2}}\Big]
\le   C_p\left\{\sum_{k=1}^n\Sbep\big[\big|X_k\big|^p\big]+\left(\sum_{k=1}^n\Sbep\big[X_k^2\big]\right)^{\frac{p}{2}}
+\left(\sum_{k=1}^n\Sbep\big[X_k^4\big]\right)^{\frac{p}{4}}\right\}.
\end{align*}
Hence
\begin{align}\label{eqproofth2.11}
 \Sbep \Big[\Big(\sum_{k=1}^nX_k^2\Big)^{\frac{p}{2}}\Big]
\le   C_p\left\{\sum_{k=1}^n\Sbep\big[\big|X_k\big|^p\big]+\left(\sum_{k=1}^n\Sbep\big[X_k^2\big]\right)^{\frac{p}{2}}
+\left(\sum_{k=1}^n\Sbep\big[X_k^4\big]\right)^{\frac{p}{4}}\right\}.
\end{align}
By applying  H\"older's inequality, it follows that
$$ \Sbep[X_k^4]=\Sbep\left[\Big(X_k^2\Big)^{\frac{p-4}{p-2}}\Big(|X_k|^p]\Big)^{\frac{2}{p-2}}\right]
\le \Big(\Sbep[X_k^2]\Big)^{\frac{p-4}{p-2}}\Big(\Sbep[|X_k|^p]\Big)^{\frac{2}{p-2}}, $$
which implies
\begin{equation}\label{eqproofth2.12}
\left(\sum_{k=1}^n\Sbep\big[X_k^4\big]\right)^{p/4}
\le C_p \left\{ \sum_{k=1}^n\Sbep[|X_k|^p]+\Big(\sum_{k=1}^n \Sbep[|X_k|^2]\Big)^{p/2}\right\}
\end{equation}
by some elementary calculation.  Substituting  (\ref{eqproofth2.11})  and (\ref{eqproofth2.12}) to  (\ref{eqMZ}), we conclude that
 (\ref{eqth2.2}) is also valid for $2^{l+1}< p\le 2^{l+2}$. By the induction, (\ref{eqth2.2}) proved. \hfill $\Box$

\bigskip

{\em Proof of Theorem \ref{th2c}}.  Suppose that $X_k$ is  independent to $(X_{k+1},\ldots, X_n)$ for each $k=1,\ldots, n-1$. Due to (\ref{eqproofth2.1.0}), we
have
 \begin{align*}
|T_k|^p\le   2^{2-p} |X_k|^p +   (T_{k+1}^+)^p +p X_k(T_{k+1}^+)^{p-1}.
\end{align*}
By the independence and the fact that $\cSbep[X+Y]\le \cSbep[X]+\Sbep[Y]$, it follows that
\begin{align*}
& \cSbep\left[ 2^{2-p} |X_k|^p +   (T_{k+1}^+)^p +p X_k(T_{k+1}^+)^{p-1}\big|T_{k+1}^+\right]\\
\le & 2^{2-p} \Sbep\left[|X_k|^p\right]+(T_{k+1}^+)^p+p \cSbep[X_k](T_{k+1}^+)^{p-1}
\le  2^{2-p} \Sbep\left[|X_k|^p\right]+(T_{k+1}^+)^p.
\end{align*}
So
\begin{align*}
  \cSbep\left[|T_k|^p\right]\le \cSbep\left[ 2^{2-p} |X_k|^p +   (T_{k+1}^+)^p +p X_k(T_{k+1}^+)^{p-1}\right] \le
  2^{2-p} \Sbep\left[|X_k|^p\right]+\cSbep\left[|T_{k+1}|^p\right].
\end{align*}
It follows that
$$\cSbep\left[|T_1|^p\right]\le 2^{2-p}  \sum_{k=1}^n  \Sbep\left[|X_k|^p\right]. $$
Now, (\ref{eqth2c.1}) is proved.  (\ref{eqth2c.2}) follows from (\ref{eqth2c.1}) by  considering the sequence $\{X_1,X_2,\ldots, X_n\}$ in the reverse order as $\{X_n,X_{n-1},\ldots, X_1\}$. \hfill $\Box$

%%%%%%%%%%%%%%%%%%%%%%%%%%%%%%%%%%%%%%%%%%%%%%%%%%%%%%%%%%%%%%%%%%%%%%%%%%%%%%%%%%%%%%%%%%%%%%%%%%%%%%%%%%%%%%%%%%%
%%%%%%%%%%%%% Section 3 %%%%%%%%%%%%%%%%%%%%%%%%%%%%%%%%%%%

  \section{Strong laws of large numbers under capacities}\label{secction3}

\setcounter{equation}{0}

Let $\mathcal G\subset\mathcal F$. A function $V:\mathcal G\to [0,1]$ is called a capacity if
$$ V(\emptyset)=0, \;V(\Omega)=1 \; \text{ and } V(A)\le V(B)\;\; \forall\; A\subset B, \; A,B\in \mathcal G. $$
It is called to be sub-additive if $V(A\bigcup B)\le V(A)+V(B)$ for all $A,B\in \mathcal G$  with $A\bigcup B\in \mathcal G$.

Here we only consider the capacities generated by a sub-linear expectation. Let $(\Omega, \mathscr{H}, \Sbep)$ be a sub-linear space, and  $\cSbep $  be  the conjugate expectation of $\Sbep$.
Furthermore, let us denote a pair $(\Capc,\cCapc)$ of capacities by
$$ \Capc(A):=\inf\{\Sbep[\xi]: I_A\le \xi, \xi\in\mathscr{H}\}, \;\; \cCapc(A):= 1-\Capc(A^c),\;\; \forall A\in \mathcal F, $$
where $A^c$  is the complement set of $A$.
Then
\begin{equation}\label{eq3.1} \begin{matrix}
&\Capc(A):=\Sbep[I_A], \;\; \cCapc(A):= \cSbep[I_A],\;\; \text{ if } I_A\in \mathscr H\\
&\Sbep[f]\le \Capc(A)\le \Sbep[g], \;\;\cSbep[f]\le \cCapc(A) \le \cSbep[g],\;\;
\text{ if } f\le I_A\le g, f,g \in \mathscr{H}.
\end{matrix}
\end{equation}
The corresponding Choquet integrals/expecations $(C_{\Capc},C_{\cCapc})$ are defined by
$$ C_V[X]=\int_0^{\infty} V(X\ge t)dt +\int_{-\infty}^0\left[V(X\ge t)-1\right]dt $$
with $V$ being replaced by $\Capc$ and $\cCapc$ respectively.

\begin{definition}\label{def3.1}
\begin{description}
\item{\rm (I)} A sub-linear expectation $\Sbep: \mathscr{H}\to \mathbb R$ is called to be  countably sub-additive if it satisfies
\begin{description}
  \item[\rm (e)] {\bf Countable sub-additivity}: $\Sbep[X]\le \sum_{n=1}^{\infty} \Sbep [X_n]$, whenever $X\le \sum_{n=1}^{\infty}X_n$,
  $X, X_n\in \mathscr{H}$ and
  $X\ge 0, X_n\ge 0$, $n=1,2,\ldots$;
 \end{description}
It is called   continuous if it satisfies
\begin{description}
  \item[\rm (f) ]  {\bf Continuity from below}: $\Sbep[X_n]\uparrow \Sbep[X]$ if $0\le X_n\uparrow X$, where $X_n, X\in \mathscr{H}$;
  \item[\rm (g) ] {\bf Continuity from above}: $\Sbep[X_n]\downarrow \Sbep[X]$ if $0\le X_n\downarrow X$, where $X_n, X\in \mathscr{H}$.
\end{description}

\item{\rm (II)}  A function $V:\mathcal F\to [0,1]$ is called to be  countably sub-additive if
$$ V\Big(\bigcup_{n=1}^{\infty} A_n\Big)\le \sum_{n=1}^{\infty}V(A_n) \;\; \forall A_n\in \mathcal F. $$

\item{\rm (III)}  A capacity $V:\mathcal F\to [0,1]$ is called a continuous capacity if it satisfies
\begin{description}
  \item[\rm (III1) ] {\bf Ccontinuity from below}: $V(A_n)\uparrow V(A)$ if $A_n\uparrow A$, where $A_n, A\in \mathcal F$;
  \item[\rm (III2) ] {\bf Continuity from above}: $V(A_n)\downarrow  V(A)$ if $A_n\downarrow A$, where $A_n, A\in \mathcal F$.
\end{description}
\end{description}
\end{definition}

{
\noindent{\bf Example~\ref{example1} (continued)} {\it The sub-linear expectation $\Sbep$ defined in Example~\ref{example1} is continuous from below,
and so is countably sub-additive.
If $\mathscr{H}$ is the set of all random variables and $\mathcal P$ is a weakly compact set of probability measures defined on $(\Omega, \mathcal F)$,
then  $(\Capc,\cCapc)$ is a pair of continuous capacities.
}

\begin{definition} Let $\{X_n;n\ge 1\}$ be a sequence of random variables in the sub-linear expectation space $(\Omega, \mathscr{H}, \Sbep)$. $X_1,X_2,\ldots $ are said to be independent
if $X_{i+1}$ is independent to $(X_1,\ldots, X_i)$ for each $i\ge 1$,
they are said to be negatively dependent if $X_{i+1}$ is negatively dependent
to $(X_1,\ldots, X_i)$ for each $i\ge 1$, and they are said to be identically distributed if $X_i\overset{d}= X_1$ for each $i\ge 1$.
\end{definition}

It is obvious that, if $\{X_n;n\ge 1\}$ is a sequence of  independent random variables and $f_1(x),f_2(x),\ldots\in C_{l,Lip}(\mathbb R)$,
 then $\{f_n(X_n);n\ge 1\}$ is  also a sequence of independent random variables;  if $\{X_n;n\ge 1\}$ is  a sequence of negatively dependent random variables and $f_1(x),f_2(x),\ldots\in C_{l,Lip}(\mathbb R)$
 are non-decreasing (resp.  non-increasing) functions, then $\{f_n(X_n);n\ge 1\}$ is  also a sequence of negatively dependent random variables.

For a sequence $\{X_n;n\ge 1\}$ of random variables in the sub-linear expectation space $(\Omega, \mathscr{H}, \Sbep)$,
we denote $S_n=\sum_{k=1}^n X_k$, $S_0=0$. The main purpose of
this section is to establish the following Kolmogorov type strong laws of larger numbers.

\begin{theorem} \label{th3}
\begin{description}
  \item[\rm (a)]  Let $\{X_n;n\ge 1\}$ be a sequence of negatively dependent and identically distributed  random variables.
  Suppose that $\Capc$ is countably sub-additive,    $C_{\Capc}[|X_1|]<\infty$ and $\lim_{c\to \infty} \Sbep\left[(|X_1|-c)^+\right]=0$. Then
\begin{equation}\label{eqth3.1}
\Capc\left(\Big\{\liminf_{n\to \infty}\frac{S_n}{n}< \cSbep[X_1]\Big\}\bigcup \Big\{\limsup_{n\to \infty}\frac{S_n}{n}> \Sbep[X_1]\Big\}\right)=0.
\end{equation}
\item[\rm (b)] Suppose that $\{X_n;n\ge 1\}$ is a sequence of independent and identically distributed  random variables, and $\Capc$ is continuous. If
\begin{equation}\label{eqth3.2}
\Capc\left(  \limsup_{n\to \infty}\frac{|S_n|}{n}=+\infty\Big\}\right)<1,
\end{equation}
then $C_{\Capc}[|X_1|]<\infty$.
\item[\rm (c)] Suppose that $\{X_n;n\ge 1\}$ is  a sequence of independent and identically distributed  random variables with
     $C_{\Capc}[|X_1|]<\infty$ and $\lim_{c\to \infty} \Sbep\left[(|X_1|-c)^+\right]=0$. If $\Capc$ is continuous, then
 \begin{equation}\label{eqcor3.2.1}
\Capc\left( \liminf_{n\to \infty}\frac{S_n}{n}= \cSbep[X_1]  \;  \text{ and }\;
  \limsup_{n\to \infty}\frac{S_n}{n}= \Sbep[X_1]\right)=1
\end{equation}
and
\begin{equation}\label{eqcor3.2.2}
\Capc\left( C\left\{\frac{S_n}{n}\right\}=\left[\cSbep[X_1], \Sbep[X_1]\right]\right)=1,
\end{equation}
where  $C(\{x_n\})$ denotes the cluster set of a sequence of $\{x_n\}$ in $\mathbb R$.
\end{description}
\end{theorem}

The following corollary follows from Theorem \ref{th3} immediately.
\begin{corollary}\label{cor3.1}
Suppose that $\mathscr{H}$ is a monotone class in the sense that $X\in \mathscr{H}$
whenever  $\mathscr{H}\ni X_n\downarrow X\ge 0$. Assume that  $\Sbep$ is continuous.
Let  $\{X_n;n\ge 1\}$ be a sequence of independent and identically distributed  random variables in $(\Omega, \mathscr{H}, \Sbep)$.  Then
$$  (\ref{eqth3.2}) \implies C_{\Capc}[|X_1|]<\infty \implies (\ref{eqth3.1}). $$
\end{corollary}

Because $\Capc$ may be not countably sub-additive in general, we define an  outer capacity $\outCapc$ by
$$ \outCapc(A)=\inf\Big\{\sum_{n=1}^{\infty}\Capc(A_n): A\subset \bigcup_{n=1}^{\infty}A_n\Big\},\;\; \outcCapc(A)=1-\outCapc(A^c),\;\;\; A\in\mathcal F.$$
  Then it can be shown that $\outCapc(A)$ is a countably sub-additive capacity with $\outCapc(A)\le \Capc(A)$ and the following properties:
  \begin{description}
    \item[\rm (a*)] If $\Capc$ is countably sub-additive, then  $\outCapc\equiv\Capc$.
    \item[\rm (b*)] If $I_A\le g$, $g\in \mathscr{H}$, then $\outCapc(A)\le \Sbep[g]$. Further, if $\Sbep$ is countably sub-additive, then
   \begin{equation}\label{eq3.2}
    \Sbep[f]\le \outCapc(A)\le \Capc(A)\le \Sbep[g], \;\; \forall f\le I_A\le g, f,g\in \mathscr{H}.
    \end{equation}
    \item[\rm (c*)]  $\outCapc$ is the largest countably sub-additive capacity satisfying
    the property that $\outCapc(A)\le \Sbep[g]$ whenever $I_A\le g\in \mathscr{H}$, i.e.,
    if $V$ is also a countably sub-additive capacity satisfying $V(A)\le \Sbep[g]$ whenever $I_A\le g\in \mathscr{H}$, then $V(A)\le \outCapc(A)$.
    \end{description}
   In fact, it is obvious that {\rm (c*)} implies  {\rm (a*)}.  For {\rm (b*)} and {\rm (c*)}, suppose $A\subset \bigcup_{n=1}^{\infty}A_n$, $\sum_{n=1}^{\infty}\Capc(A_n)
   \le \outCapc(A)+\epsilon/2$ with $I_{A_n}\le f_n\in\mathscr{H} $ and $\Sbep[f_n]\le \Capc(A_n)+\epsilon/2^{n+2}$.
   If $\mathscr{H}\ni f\le I_A$, then
   $ f\le \sum_{n=1}^{\infty} I_{A_n}\le \sum_{n=1}^{\infty} f_n$,
   which implies
   $$ \Sbep[f]\le \sum_{n=1}^{\infty} \Sbep[f_n]\le \sum_{n=1}^{\infty}\Capc(A_n)+\sum_{n=1}^{\infty}\epsilon/2^{n+2}\le
   \outCapc(A)+\epsilon $$
   by the countable sub-additivity of $\Sbep$. While, if $V$ is countably sub-additive, then
   $$ V(A)\le \sum_{n=1}^{\infty}V(A_n)\le \sum_{n=1}^{\infty}\Sbep[f_n]\le \sum_{n=1}^{\infty}\Capc(A_n)+\sum_{n=1}^{\infty}\epsilon/2^{n+2}
  \le \outCapc(A)+\epsilon. $$

   \begin{theorem} \label{th4}  Let $\{X_n;n\ge 1\}$ be a sequence identically distributed  random variables in $(\Omega,\mathscr{H},\Sbep)$.
  \begin{description}
  \item[\rm (a)]   Suppose that $X_1, X_2, \ldots$ are  negatively dependent with    $C_{\Capc}[|X_1|]<\infty$ and $\lim_{c\to \infty} \Sbep\left[(|X_1|-c)^+\right]=0$.
   Then
\begin{equation}\label{eqth4.1}
\outCapc\left(\Big\{\liminf_{n\to \infty}\frac{S_n}{n}< \cSbep[X_1]\Big\}\bigcup \Big\{\limsup_{n\to \infty}\frac{S_n}{n}> \Sbep[X_1]\Big\}\right)=0.
\end{equation}
\item[\rm (b)] Suppose that $X_1,X_2,\ldots$ are  independent,   $\outCapc$ is continuous and  $\Sbep$ is countably sub-additive.
If
\begin{equation}\label{eqth4.2}
\outCapc\left(  \limsup_{n\to \infty}\frac{|S_n|}{n}=+\infty\Big\}\right)<1,
\end{equation}
then $C_{\Capc}[|X_1|]<\infty$.
\end{description}
\end{theorem}

For proving the theorems, we need some properties of the sub-linear expectations and capacities.  We define an extension of $\Sbep$ on the space
 of all random variables by
$$ \extSbep[X]=\inf\{ \Sbep[Y]: X\le Y, Y\in \mathscr{H}\}. $$
Then $\extSbep$ is a sub-linear expectation on the space of all random variables, and
$$ \extSbep[X]=\Sbep [X] \;\; \forall X\in \mathscr{H}, \;\; \Capc(A)=\extSbep[I_A]\;\; \forall A\in \mathcal F. $$

We have the following properties.
\begin{lemma}\label{lem1}
 \begin{description}
   \item[\rm (P1)] If $\Sbep$ is continuous from below, then it is countably sub-additive;
   Similarly, if $\Capc$ is   continuous from below, then it is countably sub-additive;
   \item[\rm (P2)]  If $\Capc$ is   continuous from above, then $\Capc$ and $\cCapc$ are continuous;
    \item[\rm (P3)]  If $\Sbep$ is   continuous from above, then $\Sbep$ is continuous from below controlled, that  is,
    $\Sbep[X_n]\uparrow \Sbep[X]$ if $0\le X_n\uparrow X$, where $X_n, X\in \mathscr{H}$ and $\Sbep X<\infty$;
         \item[\rm (P4)]  Suppose that $\Sbep$ is countably sub-additive. If $ X\le \sum_{n=1}^{\infty} X_n$, $X,X_n\ge 0$ and $X\in \mathscr{H}$, then
      $ \Sbep[X]\le \sum_{n=1}^{\infty} \extSbep[X_n]$;
   \item[\rm (P5)] Set $\mathcal H=\{A: I_A\in \mathscr{H}\}$, then
$\Capc$ is a countably sub-additive  capacity in $\mathcal H$ if
$\Sbep$ is countably sub-additive in $\mathscr{H}$, and,
$(\Capc,\cCapc)$ is a pair of continuous capacities in $\mathcal H$ if
$\Sbep$ is continuous in $\mathscr{H}$.
 \end{description}
\end{lemma}

\begin{proof}. For {\rm (P1)}, if $0\le X\le \sum_{k=1}^{\infty}X_n$, $0\le X, X_n\in \mathscr{H}$, then
\begin{align*}
 \Sbep[X]=& \Sbep\left[\big(\sum_{k=1}^{\infty}X_k\big)\wedge X\right]= \lim_{n\to \infty} \Sbep\left[\big(\sum_{k=1}^nX_k\big)\wedge X\right]\\
 \le & \lim_{n\to \infty} \Sbep\left[ \sum_{k=1}^nX_k \right]\le \lim_{n\to \infty}\sum_{k=1}^n\Sbep\left[X_k \right]
 \le \sum_{k=1}^{\infty} \Sbep[X_k].
 \end{align*}
{\rm (P1)} is proved.

For {\rm (P2)}, it is sufficient to note that, if $A_n \uparrow A$, then $A\backslash A_n \downarrow \emptyset$ and $0\le \Capc(A)-\Capc(A_n)\le \Capc(A\backslash A_n)$.
Similarly, for {\rm (P3)}, it is sufficient to note that $X-X_n \downarrow 0$ and $0\le \Sbep[X]-\Sbep[X_n]\le \Sbep[X-X_n]$.

For {\rm (P4)}, choose $0\le Y_n\in \mathscr{H}$ such that $Y_n\ge X_n$,
$\Sbep[Y_n]\le \extSbep[X_n]+\frac{\epsilon}{2^{n+1}}$. Then
$X\le \sum_{n=1}^{\infty} Y_n.$
By the countable sub-additivity of $\Sbep$,
\begin{align*}
\Sbep[X]\le    \sum_{n=1}^{\infty} \Sbep[Y_n]\le \sum_{n=1}^{\infty}(\extSbep[X_n]+\frac{\epsilon}{2^{n+1}})
 \le \sum_{n=1}^{\infty}\extSbep[X_n]+\epsilon.
\end{align*}
 {\rm (P4)} is proved. {\rm (P5)} is obvious.
 \end{proof}

The following is the ``the convergence part'' of the   Borel-Cantelli Lemma   for a countably sub-additive capacity.
\begin{lemma} ({\em Borel-Cantelli's Lemma}) Let $\{A_n, n\ge 1\}$ be a sequence of events in $\mathcal F$.
Suppose that $V$ is a countably sub-additive capacity.   If $\sum_{n=1}^{\infty}V\left (A_n\right)<\infty$, then
$$ V\left (A_n\;\; i.o.\right)=0, \;\; \text{ where } \{A_n\;\; i.o.\}=\bigcap_{n=1}^{\infty}\bigcup_{i=n}^{\infty}A_i. $$
\end{lemma}
\begin{proof}  By the monotonicity and countable sub-additivity, it follows that
\begin{align*}
0\le V\left (\bigcap_{n=1}^{\infty}\bigcup_{i=n}^{\infty}A_i\right)
\le V\left (\bigcup_{i=n}^{\infty}A_i\right)\le \sum_{i=n}^{\infty}V\left (A_i\right)\to 0\; \text{ as } n\to\infty.
\end{align*}
\end{proof}

\begin{remark}
It is important to note that the condition that``$X$ is independent to $Y$ under $\Sbep$'' does not implies   that
``$X$ is independent to $Y$ under $\Capc$'' because the indicator functions $I\{X\in A\}$ and $I\{X\in A\}$ are not in $C_{l,Lip}(\mathbb R)$,
and also, ``$X$ is independent to $Y$ under $\Capc$'' does not implies   that
``$X$ is independent to $Y$ under $\Sbep$'' because $\Sbep$ is not an integral with respect to $\Capc$.
So, we have not ``the divergence part'' of the Borel-Cantelli Lemma.

Similarly, the conditions that ``$X$ and $Y$ are identically distributed under $\Sbep$''  and that that ``$X$ and $Y$ are identically distributed under
$\Capc$'' do not implies each other.
\end{remark}

\begin{lemma} \label{lem3} Suppose that $X\in \mathscr{H}$ and $C_{\Capc}(|X|)<\infty$.

(a) Then
\begin{equation}\label{eqlem3.1} \sum_{j=1}^{\infty} \frac{\Sbep[(|X|\wedge j)^2]}{j^2}<\infty.
\end{equation}

(b) Furthermore, if $\lim_{c\to \infty} \Sbep\left[|X|\wedge c\right]=\Sbep \left[|X|\right]$, then
\begin{equation}\label{eqlem3.2}
\Sbep [|X|]\le C_{\Capc}(|X|).
\end{equation}

(c) If $\Sbep$ is  countably sub-additive, then
\begin{equation}\label{eqlem3.3}\Sbep [|Y|]\le C_{\Capc}(|Y|), \;\; \forall Y\in\mathscr{H}
\end{equation}
and
\begin{equation}\label{eqlem3.4} \lim_{c\to \infty}\Sbep\left[(|X|-c)^+ \right]=0,\;\;
\lim_{c\to \infty} \Sbep\left[|X|\wedge c\right]=\Sbep \left[|X|\right]
\end{equation}
whenever $C_{\Capc}(|X|)<\infty$.
\end{lemma}
\begin{proof} (a) Note
\begin{align*}
 (|X|\wedge j)^2 =& \sum_{i=1}^j |X|^2I\{i-1<|X|\le i\}+j^2I\{|X|>j\}\\
 \le & \sum_{i=1}^j i^2I\{i-1<|X|\le i\}+j^2I\{|X|>j\}\\
 =& \sum_{i=0}^{j-1}(i+1)^2I\{|X|>i\}-\sum_{i=1}^{j}i^2I\{|X|>i\} +j^2I\{|X|>j\}\\
 \le &1+\sum_{i=1}^{j-1}(2i+1)I\{|X|>i\} \\
 \le  & 1+3\sum_{i=1}^{j}iI\{|X|>i\}.
 \end{align*}
 So,
 $$ \Sbep\left[(|X|\wedge j)^2\right]=\extSbep\left[(|X|\wedge j)^2\right]\le 1+3\sum_{i=1}^{j}i\Capc(|X|>i), $$
by the (finite) sub-additivity of $\extSbep$. It follows that
\begin{align*}
&\sum_{j=1}^{\infty} \frac{\Sbep[(|X|\wedge j)^2 ]}{j^2}
\le \sum_{j=1}^{\infty} \frac{1+3\sum_{i=1}^ji\Capc(|X|>i)}{j^2} \\
\le & 2+3\sum_{i=1}^{\infty}i\Capc(|X|>i)\sum_{j=i+1}^{\infty} \frac{1}{j^2}\le 2+3\sum_{i=1}^{\infty} \Capc(|X|>i)
\le 2+3 C_{\Capc}(|X|).
\end{align*}
(\ref{eqlem3.1}) is proved.

(b) For $n>2$, note
\begin{align*}
&|X|\wedge n = \sum_{i=1}^n|X|I\{i-1<|X|\le i\}+nI\{|X|> n\}\\
\le &\sum_{i=1}^n i\big(I\{|X|> i-1\}-I\{|X|> i\}\big)+nI\{|X|> n\}
\le  1+\sum_{i=1}^nI\{|X|> i\}.
\end{align*}
It follows that
$$ \Sbep\big[(|X|\wedge n\big]=\extSbep\big[(|X|\wedge n\big]\le 1+\sum_{i=1}^n\Capc(|X|\ge i)\le 1+\int_0^n \Capc(|X|\ge x)dx. $$
Taking $n\to \infty$ yields
$$ \Sbep\big[|X|\big]=\lim_{n\to \infty}\Sbep\big[|X|\wedge n\big] \le 1+  C_{\Capc}(|X|). $$
By considering $|X|/\epsilon$ instead of $|X|$, we have
$$ \Sbep\left[\frac{|X|}{\epsilon}\right] \le 1+  C_{\Capc}\left(\frac{|X|}{\epsilon}\right)=1+\frac{1}{\epsilon} C_{\Capc}(|X|). $$
That is
$ \Sbep\big[|X|\big] \le \epsilon +  C_{\Capc}(|X|). $
Taking $\epsilon\to 0$ yields (\ref{eqlem3.2}).

(c) Now, from the fact that
$|Y|\le 1+\sum_{i=1}^{\infty}I\{|Y|\ge i\}, $
by the countable sub-additivity of $\Sbep$ and Property (P4) in Lemma \ref{lem1}, it follows that
$$\Sbep[|Y|]\le 1+\sum_{i=1}^{\infty}\extSbep[I\{|Y|\ge i\}]= 1+\sum_{i=1}^{\infty}\Capc(|Y|\ge i)\le 1+C_{\Capc}(|Y|). $$
And then (\ref{eqlem3.3}) is proved by the same argument in (b) above.

Letting $Y=(X-c)^+$ in (\ref{eqlem3.3}) yields
$$\Sbep\big[(|X|-c)^+\big]\le  C_{\Capc}\big((|X|-c)^+\big)=\int_c^{\infty} \Capc(|X|\ge x)dx\to 0 \; \text{ as } c\to \infty. $$
And so
$$ 0\le \Sbep[|X|]-\Sbep\big[|X|\wedge c\big]\le \Sbep\big[(|X|-c)^+\big]\to 0 \; \text{ as } c\to \infty. $$
(\ref{eqlem3.4}) is proved.
\end{proof}

\bigskip

{\em Proof of Theorems \ref{th3} of \ref{th4}.} We first prove (a) of Theorem~\ref{th4}.  (a) of Theorem~\ref{th3} follows from (a) of Theorem~\ref{th4}
because $\outCapc=\Capc$ when $\Capc$ is countably sub-additive.

 Without loss of generality, we assume $\Sbep[X_1]=0$. Define
\begin{equation}\label{eqproofth4.1} f_c(x)= (-c)\vee (x\wedge c), \;\; \widehat{f}_c(x)=x-f_c(x)
\end{equation}
 and
 $$\overline{X}_j=f_j(X_j)-\Sbep[f_j(X_j)],\;\; \overline{S}_j=\sum_{i=1}^j\overline{X}_i, \;\; j=1,2,\ldots. $$
 Then $f_c(\cdot), \widehat{f}_c(\cdot)\in C_{l,Lip}(\mathbb R)$, and $\overline{X}_j$, $j=1,2,\ldots$ are negatively dependent.
  Let $\theta>1$, $n_k=[\theta^k]$. For $n_k<n\le n_{k+1}$, we have
\begin{align*}
 \frac{S_n}{n}= &\frac{1}{n}\left\{ \overline{S}_{n_{k+1}}+\sum_{j=1}^{n_{k+1}}\Sbep[f_j(X_j)]+\sum_{j=1}^n \widehat{f}_j(X_j)-\sum_{j=n+1}^{n_{k+1}}f_j(X_j)\right]\\
 \le & \frac{\overline{S}_{n_{k+1}}^+}{n_k}+\frac{\sum_{j=1}^{n_{k+1}}|\Sbep[f_j(X_1)]|}{n_k}
 +\frac{ \sum_{j=1}^{n_{k+1}}| \widehat{f}_j(X_j)|}{n_k}\\
& +\frac{\sum_{j=n_k+1}^{n_{k+1}}\big\{f_j^+(X_j) -\Sbep[f_j^+(X_j)]\big\}}{n_k}
+\frac{\sum_{j=n_k+1}^{n_{k+1}}\big\{f_j^-(X_j) -\Sbep[f_j^-(X_j)]\big\}}{n_k}\\
& +\frac{(n_{k+1}-n_k)\Sbep |X_1|}{n_k}\\
=:&(I)_k+(II)_k+(III)_k+(IV)_k+(V)_k+(VI)_k.
 \end{align*}
 It is obvious that
 $$\lim_{k\to \infty} (VI)_k = (\theta-1)\Sbep [|X_1|] \le (\theta-1)C_{\Capc}(|X_1|)$$
 by Lemma~\ref{lem3} (b).

For $(I)_k$, applying (\ref{eqcor2.3}) yields
\begin{align*} \Capc\left(\overline{S}_{n_{k+1}}\ge \epsilon n_k\right)
\le & \frac{ \sum_{j=1}^{n_{k+1}}\Sbep\big[\overline{X}_j^2\big]}{\epsilon^2 n_k^2}
\le \frac{ 4\sum_{j=1}^{n_{k+1}}\Sbep\big[f_j^2(X_1)\big]}{\epsilon^2 n_k^2} \\
\le & \frac{4 n_{k+1}}{\epsilon^2n_k^2}+\frac{ 4\sum_{j=1}^{n_{k+1}}\Sbep\big[\big(|X_1|\wedge j)^2\big]}{\epsilon^2 n_k^2}.
\end{align*}
It is obvious that $\sum_k\frac{ n_{k+1}}{n_k^2}<\infty$. Also,
\begin{align*}
 \sum_{k=1}^{\infty}\frac{  \sum_{j=1}^{n_{k+1}}\Sbep\big[\big(|X_1|\wedge j)^2\big]}{  n_k^2}
\le &  \sum_{j=1}^{\infty} \Sbep\big[\big(|X_1|\wedge j)^2\}\big]\sum_{k: n_{k+1}\ge j}\frac{1}{n_k^2}\\
\le & C\sum_{j=1}^{\infty} \Sbep\big[\big(|X_1|\wedge j)^2\big]\frac{1}{j^2}<\infty
\end{align*}
by Lemma \ref{lem3} (a). Hence
$$ \sum_{k=1}^{\infty} \outCapc\left((I)_k\ge \epsilon\right)\le \sum_{k=1}^{\infty} \Capc\left((I)_k\ge \epsilon\right)<\infty.$$
 By the Borel-Cantelli lemma and the countable sub-additivity of $\outCapc$ , it follows that
$$ \outCapc\left(\limsup_{k\to \infty} (I)_k>\epsilon\right)=0, \;\;\forall \epsilon>0 $$
Similarly,
$$ \outCapc\left(\limsup_{k\to \infty} (IV)_k>\epsilon\right)=0,\;\; \outCapc\left(\limsup_{k\to \infty} (V)_k>\epsilon\right)=0, \;\;\forall \epsilon>0. $$

For $(II)_k$, note that by the (finite) sub-additivity,
$$|\Sbep[f_j(X_1)]|=|\Sbep[f_j(X_1)]-\Sbep X_1|\le \Sbep[|\widehat{f}_j(X_1)|]= \Sbep[(|X_1|-j)^+]\to 0. $$
It follows that
$$ (II)_k=\frac{n_{k+1}}{n_k}\frac{\sum_{j=1}^{n_{k+1}}|\Sbep[f_j(X_1)]|}{n_{k+1}}\to 0. $$

At last, we consider $(III)_k$. By the Borel-Cantelli Lemma, we will have
$$ \outCapc\left(\limsup_{k\to \infty}(III)_k>0\right)\le \outCapc\big( \{|X_j|> j\} \; i.o.\big)=0 $$
if we have shown that
\begin{equation}\label{eqproofth3.10}\sum_{j=1}^{\infty} \outCapc\big(  |X_j|> j \big)\le \sum_{j=1}^{\infty} \Capc\big(  |X_j|> j \big)<\infty.
\end{equation}
Let $g_{\epsilon}$ be a function satisfying that its derivatives of each order are bounded, $g_{\epsilon}(x)=1$ if $x\ge 1$, $g_{\epsilon}(x)=0$ if $x\le 1-\epsilon$, and $0\le g_{\epsilon}(x) \le 1$ for all $x$,
where $0<\epsilon<1$.
Then
$$ g_{\epsilon}(\cdot) \in C_{l,Lip}(\mathbb R)\;  \text{ and }\; I\{x\ge 1\}\le g_{\epsilon}(x)\le I\{x>1-\epsilon\}. $$
Hence, by (\ref{eq3.1}),
 \begin{align*}
  \sum_{j=1}^{\infty} \Capc\big(  |X_j|> j \big)
 \le & \sum_{j=1}^{\infty}\Sbep\left[g_{1/2}\big(|X_j|/ j\big)\right]= \sum_{j=1}^{\infty}\Sbep\left[g_{1/2}\big(|X_1|/ j\big)\right]
  \;\; (\text{ since } X_j\overset{d}= X_1 )\nonumber \\
 \le &  \sum_{j=1}^{\infty} \Capc\big(  |X_1|> j/2 \big)\le 1+C_{\Capc}(2|X_1|)<\infty.\nonumber
 \end{align*}
(\ref{eqproofth3.10}) is proved. So, we conclude that
$$ \outCapc\left(\limsup_{n\to \infty} \frac{S_n}{n}>\epsilon\right)=0 \;\; \forall \epsilon>0,$$
by the arbitrariness of $\theta>1$.  Hence
$$ \outCapc\left(\limsup_{n\to \infty} \frac{S_n}{n}>0\right)
=\outCapc\left(\bigcup_{k=1}^{\infty}\left\{\limsup_{n\to \infty} \frac{S_n}{n}>\frac{1}{k}\right\}\right)
\le \sum_{k=1}^{\infty}\outCapc\left( \limsup_{n\to \infty} \frac{S_n}{n}>\frac{1}{k} \right)=0.
$$
Finally,
$$  \outCapc\left(\liminf_{n\to \infty} \frac{S_n}{n}<\cSbep[X_1]\right)
=\outCapc\left(\limsup_{n\to \infty} \frac{\sum_{k=1}^{n}(-X_k-\Sbep[-X_k])}{n}>0\right)=0. $$
The proof of (\ref{eqth3.1}) is now completed.

For (b) of Theorems \ref{th3} and \ref{th4}, suppose $C_{\Capc}(|X_1|)=\infty$. Then, by (\ref{eq3.1}),
\begin{align}\label{eqproofth3.11}
 \sum_{j=1}^{\infty}\Sbep\left[g_{1/2}\big(\frac{|X_j|}{Mj}\big)\right]
=&\sum_{j=1}^{\infty}\Sbep\left[g_{1/2}\big(\frac{|X_1|}{Mj}\big)\right]\;\; (\text{ since } X_j\overset{d}= X_1 )\\
\ge &\sum_{j=1}^{\infty} \Capc\big(|X_1|> M j)=\infty, \;\; \forall M>0.\nonumber
\end{align}
For any $l\ge 1$,
\begin{align*}
&\cCapc\left(\sum_{j=1}^n g_{1/2}\big(\frac{|X_j|}{Mj}\big)<l\right)
 =\cCapc\left(\exp\Big\{-\frac{1}{2}\sum_{j=1}^n g_{1/2}\big(\frac{|X_j|}{Mj}\big)\Big\}>e^{-l/2}\right)\\
\le  & e^{l/2}\cSbep \left[\exp\Big\{-\sum_{j=1}^n g_{1/2}\big(\frac{|X_j|}{Mj}\big)\Big\}\right]
=e^{l/2}\prod_{j=1}^n\cSbep \left[\exp\Big\{- \frac{1}{2} g_{1/2}\big(\frac{|X_j|}{Mj}\big)\Big\}\right]
\end{align*}
by (\ref{eq3.1}) again and the independence because $0\le \exp\Big\{- \frac{1}{2}g_{1/2}\big(\frac{|x_j|}{Mj}\big)\Big\}\in C_{l,Lip}(\mathbb R)$. Applying the elementary inequality
$$ e^{-x}\le 1-\frac{1}{2} x \le e^{-x/2}, \;\; \forall 0\le x\le 1/2$$
yields
 $$\cSbep \left[\exp\Big\{- \frac{1}{2} g_{1/2}\big(\frac{|X_j|}{Mj}\big)\Big\}\right]
 \le 1-\frac{1}{4} \Sbep\Big[  g_{1/2}\big(\frac{|X_j|}{Mj}\big)\Big]
 \le \exp\left\{-\frac{1}{4} \Sbep\Big[  g_{1/2}\big(\frac{|X_j|}{Mj}\big)\Big]\right\}. $$
 It follows that
$$\cCapc\left(\sum_{j=1}^n g_{1/2}\big(\frac{|X_j|}{Mj}\big)<l\right)\le
\exp\left\{-\frac{1}{4} \sum_{j=1}^n \Sbep\Big[  g_{1/2}\big(\frac{|X_j|}{Mj}\big)\Big]\right\}\to 0\;\;\text{ as }  n\to \infty, $$
by (\ref{eqproofth3.11}). So
$$
 \Capc\left(\sum_{j=1}^n g_{1/2}\big(\frac{|X_j|}{Mj}\big)> l\right)\to 1\;\;\text{ as }  n\to \infty.
$$

If $\Capc$ is continuous as assumed in Theorem \ref{th3}, then $\Capc\equiv \outCapc$. If $\Sbep$ is countably sub-additive as assumed in Theorem \ref{th4}, then
$$\outCapc(|X|\ge c)\le \Capc(|X|\ge c)\le\Sbep\left[g_{\epsilon}(|X|/c)\right]\le \outCapc\big(|X|\ge c(1-\epsilon)\big), $$
by (\ref{eq3.1}) and (\ref{eq3.2}). In either case, we have
 $$
 \outCapc\left(\sum_{j=1}^n g_{1/2}\big(\frac{|X_j|}{Mj}\big)>l/2\right)
 \ge \Capc\left(\sum_{j=1}^n g_{1/2}\big(\frac{|X_j|}{Mj}\big)> l\right)\to 1\;\;\text{ as }  n\to \infty.
 $$

Now, by the continuity of $\outCapc$,
\begin{align*}
&
\outCapc\left(\limsup_{n\to\infty} \frac{|X_n|}{n}>\frac{M}{2}\right)=\outCapc\left(\big\{\frac{|X_j|}{Mj}> \frac{1}{2}\big\} \;\;i.o.\right)
\ge
\outCapc\left(\sum_{j=1}^{\infty} g_{1/2}\big(\frac{|X_j|}{Mj}\big)=\infty\right)\\
=&\lim_{l\to \infty} \outCapc\left(\sum_{j=1}^{\infty} g_{1/2}\big(\frac{|X_j|}{Mj}\big)> l/2\right)
=\lim_{l\to \infty} \lim_{n\to \infty} \outCapc\left(\sum_{j=1}^n g_{1/2}\big(\frac{|X_j|}{Mj}\big)> l/2\right)=1.
\end{align*}
On the other hand,
$$\limsup_{n\to\infty} \frac{|X_n|}{n}\le \limsup_{n\to\infty}\Big(\frac{|S_n|}{n}+\frac{|S_{n-1}|}{n}\Big)\le 2\limsup_{n\to\infty} \frac{|S_n|}{n}.
$$
It follows that
$$ \outCapc\left(\limsup_{n\to\infty} \frac{|S_n|}{n}>m\right)=1,\;\; \forall m>0. $$
Hence
$$ \outCapc\left(\limsup_{n\to\infty} \frac{|S_n|}{n}=+\infty\right)=\lim_{m\to \infty} \outCapc\left(\limsup_{n\to\infty} \frac{|S_n|}{n}>m\right)=1, $$
which  contradict (\ref{eqth3.2}) and (\ref{eqth4.2}). So, $C_{\Capc}(|X_1|)<\infty$.

Finally, we consider (c) of Theorem \ref{th3}.  For (\ref{eqcor3.2.1}), we first show that
\begin{equation}\label{eqproofcor3.2.1}
\Capc\left(\frac{S_n}{n}>\Sbep[X_1]-\epsilon\right)\to 1, \;\; \forall \epsilon>0.
\end{equation}
Let $f_c(x)$ and $\widehat{f}_c(x)$ be defined as in (\ref{eqproofth4.1}). Then
\begin{align*}
\Capc\left(\frac{|\sum_{k=1}^n \widehat{f}_c(X_k)|}{n}>\epsilon\right)
\le \frac{\sum_{k=1}^n \Sbep[|\widehat{f}_c(X_k)|]}{\epsilon n}\le \frac{\Sbep[(|X_1|-c)^+]}{\epsilon}\to 0
\;\; \text{ as } c\to \infty,
\end{align*}
$\Sbep[X_1]-\Sbep[f_c(X_1)]\to 0 $ as $c\to\infty$, and by Theorem \ref{th2c},
\begin{align*}
&\cCapc\left(\frac{ \sum_{k=1}^n f_c(X_k)}{n}\le \Sbep[f_c(X_1)]-\epsilon\right)
= \cCapc\left(  \sum_{k=1}^n \big(- f_c(X_k)-\cSbep[-f_c(X_k)]\big)\ge n\epsilon\right)\\
\le & \frac{\cSbep\left[\left| \left( \sum_{k=1}^n \big(- f_c(X_k)-\cSbep[-f_c(X_k)]\big)\right)^+\right|^2\right]}{n^2\epsilon^2}
\le  2\frac{\Sbep\left[\big(- f_c(X_1)-\cSbep[-f_c(X_1)]\big)^2\right]}{n\epsilon^2} \\
\le & \frac{ 2(2c)^2}{n\epsilon^2}\to 0 \;\; \text{ as } n\to \infty.
\end{align*}
Then (\ref{eqproofcor3.2.1}) is proved. By considering $\{-X_n;n\ge 1\}$  instead, from (\ref{eqproofcor3.2.1}) we have
\begin{equation}\label{eqproofcor3.2.2}
\Capc\left(\frac{S_n}{n}\le\cSbep[X_1]+\epsilon\right)\to 1, \;\; \forall \epsilon>0.
\end{equation}
Note the independence. We conclude that
\begin{align*}
&\Capc\left(\frac{S_n}{n}<\cSbep[X_1]+\epsilon \; \text { and }\;  \frac{S_{n^2}-S_n}{n^2-n}>\Sbep[X_1]-\epsilon\right) \\
\ge & \Sbep\left[\phi\left(\frac{S_n}{n}-\cSbep[X_1]\right)\phi\left(\Sbep[X_1]-\frac{S_{n^2}-S_n}{n^2-n}\right)\right]\\
\ge & \Sbep\left[\phi\left(\frac{S_n}{n}-\cSbep[X_1]\right)\right]
\cdot \Sbep\left[\phi\left(\Sbep[X_1]-\frac{S_{n^2}-S_n}{n^2-n}\right)\right]\\
\ge & \Capc\left(\frac{S_n}{n}<\cSbep[X_1]+\frac{\epsilon}{2}\right)\cdot \Capc\left(  \frac{S_{n^2}-S_n}{n^2-n}>\Sbep[X_1]-\frac{\epsilon}{2}\right) \to 1, \;\; \forall \epsilon>0,
\end{align*}
where $\phi(x)\in C_{l,Lip}(\mathbb R)$ is a  function such that $I\{x\le \epsilon\}\ge \phi(x)\ge I\{x\le \epsilon/2\}$. Now, by (\ref{eqth3.1}) and the continuity of $\Capc$,
\begin{align*}
& \Capc\left( \liminf_{n\to \infty}\frac{S_n}{n}\le \cSbep[X_1]+\epsilon  \;  \text{ and }\;
  \limsup_{n\to \infty}\frac{S_n}{n}\ge  \Sbep[X_1]-\epsilon\right) \\
\ge & \Capc\left( \liminf_{n\to \infty}\frac{S_n}{n}\le \cSbep[X_1]+\epsilon  \;  \text{ and }\;
  \limsup_{n\to \infty}\frac{S_{n^2}-S_n}{n^2-n}\ge  \Sbep[X_1]-\epsilon\right) \\
  \ge & \Capc\left(\frac{S_n}{n}<\cSbep[X_1]+\epsilon \; \text { and }\;  \frac{S_{n^2}-S_n}{n^2-n}>\Sbep[X_1]-\epsilon\;\; i.o.\right)\\
\ge & \limsup_{n\to \infty} \Capc\left(\frac{S_n}{n}<\cSbep[X_1]+\epsilon \; \text { and }\;  \frac{S_{n^2}-S_n}{n^2-n}>\Sbep[X_1]-\epsilon\right)=1, \;\; \forall \epsilon>0.
\end{align*}
B the continuity of $\Capc$ again,
$$\Capc\left( \liminf_{n\to \infty}\frac{S_n}{n}\le \cSbep[X_1]   \;  \text{ and }\;
  \limsup_{n\to \infty}\frac{S_n}{n}\ge  \Sbep[X_1] \right)=1, $$
which, together with (\ref{eqth3.1}) implies (\ref{eqcor3.2.1}).

Finally, note
$$ \frac{S_n}{n}-\frac{S_{n-1}}{n-1}=\frac{X_n}{n}-\frac{S_{n-1}}{n-1}\frac{1}{n}\to 0\;\; a.s. \Capc. $$
It can be verified that (\ref{eqcor3.2.1}) implies (\ref{eqcor3.2.2}). \hfill $\Box$

\bigskip
{\em Proof of Corollary~\ref{cor3.1}}. It is sufficient to  note the facts that $\Capc (A)=\Sbep[I_A]$ is continuous in $\mathcal H=\{A, I_A\in \mathscr{H}\}$ and
all events we consider are in $\mathcal H$ because $\mathscr{H}$ is monotone and  $I\{x\ge 1\}=\lim_{\epsilon\to 0} g_{\epsilon}(x)$. \hfill $\Box$

\Acknowledgements{This work was supported  by National Natural Science Foundation of China (Grant No. 11225104) and the Fundamental Research Funds for the Central Universities.}


\end{document}